\documentclass[a4paper,11pt,reqno]{smfart}
\usepackage{amssymb,amsmath,mathrsfs,graphics,graphicx,mathptm,float}
\usepackage[T1]{fontenc}
\usepackage[english, french]{babel}
\usepackage[all]{xy}
\theoremstyle{plain}
\newtheorem{theorem}{Theorem}[section]
\newtheorem{proposition}[theorem]{Proposition}

\newtheorem{corollary}[theorem]{Corollary}

\theoremstyle{definition}

\newtheorem{example}[theorem]{Example}
\newtheorem{remark}[theorem]{Remark}

\usepackage[colorlinks, linktocpage, 
citecolor = blue, linkcolor = blue]{hyperref}

\addtocounter{section}{0}             
\numberwithin{equation}{section}

\begin{document}
\selectlanguage{english}

\title{Codimension one holomorphic foliations on $\mathbb P^n_{\mathbb C}$: problems in complex geometry.}

\author{Dominique Cerveau}

\address{Institut Universitaire de France et IRMAR, Campus de Beaulieu, 35042 Rennes Cedex France.}
\email{dominique.cerveau@univ-rennes1.fr}

\maketitle

\begin{flushright}
{\sl En l'honneur de H. Hironaka pour son $80^{\hbox{\rm \tiny{\`eme}}}$ anniversaire}
\end{flushright}

\medskip\medskip\medskip

 \begin{abstract} 
After a short review on foliations, we prove that a codimension $1$ holomorphic foliation on $\mathbb P^3_{\mathbb C}$ with simple singularities is given by a closed rational $1$-form. The proof uses Hironaka-Matsumura prolongation theorem of formal objects.

\noindent{\it Keywords. --- algebraic foliations, reduction of singularities}

\noindent{\it 2010 Mathematics Subject Classification. --- 37F75}
\end{abstract}

\section*{Preliminaries} Let $M$ be a complex compact connected manifold of dimension $n\geq 2$. A codimension $1$ singular holomorphic foliation $\mathcal F$ on $M$ is given by a covering by open subsets $(V_{j})_{j\in J}$ and a collection of integrable holomorphic $1$-forms $\omega_{j}$ on $V_{j}$, $\omega_{j} \wedge d\omega_{j}=0$, having codimension $\geq2$ zero sets such that on each non empty intersection $V_{j} \cap V_{k}:$ 
 
 (*)\qquad $\omega_{j}=g_{jk}.\omega_{k} \ \ \hbox{\rm with}\ \ g_{jk} \in \mathcal O^{*} (V_{j} \cap U_{k}).$
 
 Let $\mathrm{Sing}\,\omega_{j}:= \{p \in V_{j}\ ,\ \omega_{j} (p)=0\}$ be the singular set of $\omega_{j}$. Condition (*) implies that $\mathrm{Sing}\,\mathcal F:=\displaystyle\cup_{j\in J}\,\mathrm{Sing}\,  \omega_{j}$ is a codimension $\geq 2$ analytic subset of $M$, the singular set of $\mathcal F$.
 
 In the special case where $M$ is a projective manifold and $\mathcal F$ a foliation as above, we can associate to $\mathcal F$ a meromorphic $1$-form $\omega$ in the following way. We take a rational vector fields $Z$ on $M$, not tangent to $\mathcal F$, that is $h_{j}=i_{Z_{\vert U_j}} \omega_{j}\not\equiv 0$; the meromorphic $1$-form $\omega$ defined on $V_{j}$ by $\omega_{\vert U_j}=\omega_{j}/h_{j}$ is global and integrable. In this case we will say that $\omega$ defines $\mathcal F$.
 
 There is another interesting very special case: the case $M=\mathbb P^n_{\mathbb C}$, the $n$ dimensional complex projective space. In that context, we have a theorem of Chow-type. Denote by $\pi\colon\mathbb C^{n+1}\setminus\{0\} \to \mathbb P^n_{\mathbb C}$ the natural projection, and consider $\pi^{-1} \mathcal F$ the pull-back of $\mathcal F$ by $\pi$; with the previous notations, $\pi^{-1} \mathcal F$ is defined by the $1$-form $\pi^* \omega_{j}$ on $\pi^{-1} (U_{j})$. Recall that, for $n\geq 2$, we have $H^1 (\mathbb C^{n+1} \setminus\{0\}, \mathcal O^{*})=\{1\}$: it is a result due to Cartan (\cite{Cartan}). As a consequence, there exists a global holomorphic $1$-form $\omega$ on $\mathbb C^{n-1}\setminus\{0\}$ which defines $\pi^{-1} \mathcal F$ on $\mathbb C^{n+1}\setminus\{0\}$.
 
 By Hartog's prolongation theorem $\omega$ can be extended holomorphically at 0. By construction we have $i_{R} \omega=0$, where $R$ is the Euler (or radial) vector fields:
 $$R=\displaystyle\mathop{\Sigma}^{n}_{i=0} z_{j} \frac{\partial}{\partial z_{j}}.$$
This fact and the integrability condition imply that $\omega$ is colinear to an integrable homogeneous $1$-form $\omega_{\nu+1}=\displaystyle\mathop{\Sigma}^{n}_{i=0} A_{i} (z) dz_{i}$, $A_{i}$ homogenenous polynomials of degree $\nu+1$, $\mathrm{gcd} (A_{0},\cdots , A_{n})=1$ ({\it i.e.} $\mathrm{cod}\,\mathrm{Sing}\, \omega_{\nu+1}\geq 2$). This is the so-called foliated Chow's Thoerem:

\begin{theorem} 
To any codimension $1$ holomorphic foliation $\mathcal F$ on $\mathbb P_{\mathbb C}^n$ is asociatied an homogeneous integrable $1$-form $\omega_{\nu+1}$ on $\mathbb C^{n+1}$ defining $\pi^{-1} \mathcal F$ with
$\mathrm{cod}\,\mathrm{Sing}\,\omega_{\nu+1} \geq 2$.
\end{theorem}

By definition the integer $\nu$ is the degree of the foliation $\mathcal F$. The homogeneous $1$-form $\omega_{\nu+1}$ is well defined up to multiplication by non zero complex number.

\begin{remark}
Denote by $U_{i}:=\{z_{i}=1\}\subset \mathbb P^n_{\mathbb C}$ the usual affine charts associated to the projective coordinates $(z_0:\cdots:z_{n})$. Then $\omega_{\nu+1_{\vert U_j}}=\omega_{j}$ is a polynomial $1$-form on $U_{j} \simeq \mathbb C^n$ which can be extended meromorphically to $\mathbb P^n_{\mathbb C}$. 
\end{remark}

We have the following facts; if $\mathcal F$ is a foliation of degree $\nu$ on $\mathbb P^n_{\mathbb C}$ then:
\begin{itemize}
\item the integer $\nu=\deg\mathcal F$ is exactly the number of tangencies of $\mathcal F$ with a generic line $L$, that is the number of points $m \in L$ where $L$ is not transverse to $\mathcal F$ (if~$m\in U_{j}$ then $L$ is "contained" in the kernel of the linear form $\omega_{j} (m)$).

\item the set $\mathrm{Sing}\,\mathcal F$ has non trivial components of codimension $2$: points in $\mathbb P^2_{\mathbb C}$, curves in $\mathbb P^3_{\mathbb C} \ldots$ In particular, there are no non singular codimension $1$ foliations on~$\mathbb P^n_{\mathbb C}$ except for $n=1$. This fact can be proved by using De Rham-Saito division lemma \cite{Saito}.
\end{itemize}

If $\mathcal F$ is a codimension $1$ foliation on $M$, the leaves of $\mathcal F$ are, by defintion, the leaves (maximal integral immersed manifolds) of the regular foliation $\mathcal F_{\vert M\setminus\mathrm{Sing}\,\mathcal F}$.

An exciting problem is to give the description of the spaces $\mathcal F(n;d)$ of codimension $1$ foliations of degree $d$ on $\mathbb P^n_{\mathbb C}$, in particular the irreducible components of these spaces. For $n=2$ the sets $\mathcal F (2;d)$ are Zariski open sets in some projective spaces and the consistency of the problem appears in dimension $\geq 3$.

A second problem consists, for each given irreducible component of $\mathcal F(n;d)$, in the description of the leaves of generic elements of that component.

\section{Some examples and known facts.}
There are many examples of foliations without singularities, in particular on tori, Hopf manifolds etc. Regular foliations on compact complex surfaces are classified by Brunella (\cite{Brunella}). Here we focus on foliations in $\mathbb P^n_{\mathbb C}$. As we have seen above, such foliations are singular.

\begin{example} Foliations of degree 0 on $\mathbb P^n_{\mathbb C}$. 

Such foliations are pencils of hyperplanes. Up to conjugacy by Aut $\mathbb P^n_{\mathbb C}$, the group of automorphisms of  $\mathbb P^n_{\mathbb C}$, there is one model, the foliation $\mathcal F_{0}$ given by the homogeneous $1$-form $z_{0} dz_{1}-z_{1} dz_{0}$. Note that $\mathcal F_{0}$ is also given by the global closed $1$-form $\frac{dz_{0}}{z_{0}}-\frac{dz_{1}}{z_{1}}$. The singular locus of $\mathcal F_{0}$ is the linear space $\{z_{0}=z_{1}=0\}\simeq \mathbb P^{n-2}_{\mathbb C}$ and the closure of the leaves are hyperplanes $z_{0}/z_{1}=$cste. Remark also that, by blowing-up the singular locus, we obtain a regular foliation on the blow-up of $\mathbb P^n_{\mathbb C}$.

The space $\mathcal F(n;0)$ is isomorphic to the Grassmanian of $(n-2)$-linear subspaces of $\mathbb P^n_{\mathbb C}$.
\end{example}

\begin{example} Foliations of degree 1 on $\mathbb P^2_{\mathbb C}$.

A generic element of $\mathcal F(2;1)$ is given in a good chart $\{(x,y)\}\simeq \mathbb C^2$ by the linear $1$-form $\lambda y dx-xdy$, $\lambda \in \mathbb C$. The leaves are parametrized by $$\mathbb  C\ni t\mapsto (x_{0} e^t, y_{0} e^{\lambda t}) \in \mathbb C^2 \subset \mathbb P^2_{\mathbb C}.$$ If $\lambda \in \mathbb Q$, the closure of the leaves are rational algebraic curves (of type $x^p y^q=$cste), and if $\lambda\in \mathbb C\setminus\mathbb Q$ the leaves are transcendental Pfaffian sets.

All foliations of degree $1$ on $\mathbb P^2_{\mathbb C}$ are given by a closed rational $1$-form ($\lambda \frac{dx}{x}-\frac{dy}{y}$ in the generic case). The set $\mathcal F(2;1)$ can be identified to a Zariski open set in the projective space $\mathbb P^7_{\mathbb C}$.
\end{example}

\begin{example} The set $\mathcal F (n;1)$, $n\geq 3$.

For $n\geq 3$, the set $\mathcal F(n;1)$ has two irreductible components corresponding to the following alernative; if $\mathcal F \in \mathcal F (n;1)$:

(*) either there exists a linear map $F\colon\mathbb P^n_{\mathbb C} \dashrightarrow \mathbb P^2_{\mathbb C}$ and $\mathcal F_{0}\in \mathcal F (2;1)$ such that $\mathcal F=F^{-1}\mathcal F_{0}$.

(**) or in a good affine chart $\mathbb C^n\subset \mathbb P^n_{\mathbb C}$, $\mathcal F$ is given by the $1$-form $\omega=dP$, where~$P$ is polynomial of degree 2. The leaves are the level sets of $P$.

In each of these two cases $\mathcal F$ is given by a closed rational $1$-form.
\end{example}

\begin{example}\label{ex:4} Quadratic foliations on $\mathbb P^n_{\mathbb C}$, $n\geq 3$.

The description of $\mathcal F(n;2)$, $n\geq 3$, is a little bit  more difficult; $\mathcal F (n;2)$ has six irreductible components (\cite{CerveauLinsNeto}) and we have the following alternative. If $\mathcal F \in \mathcal F(n;2)$, $n\geq 3$, then:

(*) either there exists a linear map $F\colon \mathbb P^n_{\mathbb C}\dashrightarrow \mathbb P^2_{\mathbb C}$ and a foliation $\mathcal F_{0}\in \mathcal F(2;2)$ on~$\mathbb P^2_{\mathbb C}$ such that $\mathcal F=F^{-1} \mathcal F_{0}$, the pull-back of $\mathcal F_{0}$ by $F$ (it corresponds to one component of $\mathcal F(n;2)$).

(**) or $\mathcal F$ is defined by a closed rational $1$-form. This second part of the alternative gives $5$ components.

One of the component is a $\mathrm{Aut}\,\mathbb P^n_{\mathbb C}$-orbit of the so-called "exceptional foliation"; this means in particular that there exists quadratic stable foliations. In fact for any $d$ the set $\mathcal F(n;d)$, $n\geq 3$, contains stable foliations (\cite{CACGLN}).
\end{example}

\begin{example} Foliations associated to closed meromorphic $1$-forms.

To each meromorphic closed $1$-form $\omega$ on $\mathbb P^n_{\mathbb C}$ is associated a codimension $1$-holomorphic foliation. Recall that such a closed form has a decomposition:
$$\omega=\Sigma \lambda_{i} \frac{d f_{i}}{f_{i}}+dh$$
where the $\lambda'_{i}s$ are complex numbers (the residues or periods) and the $f'_{i}s$ and $h$ are rational functions. The leaves are (outside the singular set of the foliation) the connected components of the "level sets" of the multivalued function $\Sigma \lambda_{i} \log f_{i}+h$. There are many deep questions concerning the nature of these leaves in relation with topology, number theory, hyperbolic geometry....

As it can be seen in \cite{CalvoAndrade}, \cite{LinsNeto}, for each degree $d$, these are several irreducible components of $\mathcal F(n;d)$, $n\geq3$, whose generic elements correspond to foliations given by closed $1$-forms.
\end{example}

\begin{example}\label{ex:6} Degree 3 foliations.

The explicit decomposition in irreducible components of the space $\mathcal F(n;3)$, $n\geq~3$, is not known. Nevertheless there is a qualitative description of the elements of $\mathcal F(n;3)$; in fact we have an alternative quasi-similar to Example \ref{ex:4}: for $\mathcal F\in \mathcal F(n;3)$

(*) either there exist $F\colon\mathbb P^n_{\mathbb C}\dashrightarrow \mathbb P^2_{\mathbb C}$ rational and $\mathcal F_{0}$ a foliation on $\mathbb P^2_{\mathbb C}$ such that $\mathcal F= F^{-1} \mathcal F_{0}$, 

(**) or $\mathcal F$ is defined by a closed rational $1$-form. 

This alternative is the consequence of the two papers \cite{CerveauLinsNeto2} and \cite{LPT}; the differen\-ce with Example \ref{ex:4} is that there is no control of the degrees of the rational map $F$ and the foliation $\mathcal F_{0}$.

There exist foliations in degree $>3$ on $\mathbb P^n_{\mathbb C}$, $n\geq 3$, which don't satisfy the alternative of Example \ref{ex:6}. We will now speak a little bit of families of such examples, the so-called transversally projective foliations.
\end{example}

\begin{example} Transversally projective foliations.

To such a foliation $\mathcal F_{0}$ is associated a "$\mathrm{sl}(2;\mathbb C)$ 3-uple" $(\omega_{0}, \omega_{1}, \omega_{2})$ of rational $1$-forms on $\mathbb P^n_{\mathbb C}$ satisfying:

(*) $\mathcal F_{0}$ is given by $\omega_{0}$,

(**) the $\omega_{i}$ verify the Maurer-Cartan conditions:
$$ d\omega_{0}=\omega_{0} \wedge \omega_{1},\hspace{2cm} 
 d\omega_{1}=\omega_{0}\wedge \omega_{2},\hspace{2cm}
 d\omega_{2}=\omega_{1}\wedge \omega_{2}.$$

Remark that a foliation given by a closed rational $1$-form $\omega_{0}$ is a special case of transversally projective foliation (take $\omega_{1}=\omega_{2}=0$).  The Maurer-Cartan conditions imply the integrability of the unfolding:
\begin{center}
$\Omega=dt+\omega_{0}+t \omega_{1}+\frac{t^2}{2} \omega_{2},\,\,\,\,\,\, t\in \mathbb C \subset \mathbb P^1_{\mathbb C},$
\end{center}
which defines a "Riccati-foliation" on $\mathbb P^n_{\mathbb C}\times \mathbb P^1_{\mathbb C}$. We see that the restriction $\Omega$ to $t=0$ gives the foliation $\mathcal F_{0}$ and the restriction to $t=\infty$ gives, in the case $\omega_{2}\not\equiv 0$, a new foliation $\mathcal F_{2}$ associated to $\omega_{2}$.

Note  also that to an ordinary Riccati differential equation $\frac{dy}{dx}=a(x) y^2+b(x) y+c(x)$, $a,b,c\in \mathbb C(x)$ is associated a transversally projective foliation on $\mathbb P^3_{\mathbb C}$ given (in an affine chart) by:
$$\omega_{0}=dz+\omega'_{0} +z\omega'_{1}+\frac{z^2}{2} \omega'_{2}$$
with, denoting by $L$ the Lie derivative:
$$\omega'_{0}=dy-(a(x) y^2+b(x) y+c(x))dx,\hspace{1cm}\omega'_{1}=L_{\frac{\partial}{\partial y}} \omega'_{0},\hspace{1cm}\omega'_{2}=L_{\frac{\partial}{\partial y}} \omega'_{1}.$$

Here the corresponding $\mathrm{sl}(2;\mathbb C)$ 3-uple is $(\omega_{0},\omega_{1}=\omega'_{1}+z \omega'_{2}, \omega_{2}=\omega'_{2})$. We have the following fact: there are explicit constructions of transversally projective foliations on $\mathbb P^3_{\mathbb C}$ associated to some special rational Hilbert-modular surfaces (\cite{CLNLPT}). These foliations are not defined by closed meromorphic $1$-forms and are not rational pull-back of foliations on $\mathbb P^2_{\mathbb C}$ (\emph{see} \cite{CLNLPT}).

All known foliations $\mathcal F$ on $\mathbb P^n_{\mathbb C}$, $n\geq 3$, satisfy the following alternative I: $\mathcal F$ is 

(*) either transversally projective

(**) or a rational pull-back of a foliation $\mathcal F_{0}$ on $\mathbb P^2_{\mathbb C}$.

We don't know if the previous alternative I is always satisfied or if there exist other types of foliations on $\mathbb P^n_{\mathbb C}$. 

It is possible to prove that a transversally projective foliation $\mathcal F$ on $\mathbb P^n_{\mathbb C}$ has  an invariant hypersurface $X\subset \mathbb P^n_{\mathbb C}$ (\emph{see} \cite{CLNLPT}): $X\setminus\mathrm{Sing}\mathcal F$ is a leaf of the regular foliation $\mathcal F_{\vert\mathbb P_{\mathbb C}^n \setminus\mathrm{Sing}\, \mathcal F}$. For example if $\mathcal F$ is given by a closed $1$-form $\omega$, then the divisor of the poles of $\omega$ is such an invariant hypersurface.

A contrario, generic foliations on $\mathbb P^2_{\mathbb C}$, with degree $\geq 2$, have no invariant algebraic curves (\cite{Jouanolou}). This implies that general pull-back foliations $F^{-1}\mathcal F_{0}$, $F\colon \mathbb P^n_{\mathbb C} \dashrightarrow\mathbb P^2_{\mathbb C}$, don't have invariant hypersurfaces.

The following conjecture due to Brunella says that a foliation on $\mathbb P^n_{\mathbb C}$ either is a rational pull-back of a foliation on $\mathbb P^2$ or has an invariant algebraic hypersurface (alternative II). As we have seen alternative I implies alternative II and alternative I is satisfied in small degree $(\leq 3)$ for foliations on $\mathbb P^n_{\mathbb C}$. We mention that alernative I is always satisfied for foliations on $\mathbb P^n_{\Bbbk}$, where $\Bbbk$ is a field of positive characteristic (\cite{CLNLPT}).
\end{example}

\section{Reduction of singularities for codimension one foliations in dimension $\leq 3$.}

In dimension $2$, Seidenberg gives in \cite{Seidenberg} the first statement concerning the reduction of singularities. For germs of analytic subsets $X$ in $\mathbb C^n,0$ we know, following Hironaka, that after suitable blow-up $\pi\colon\widetilde{\mathbb C}^n \rightarrow \mathbb C^n,0$ the total transform $\pi^{-1} (X)$ is locally given by the zeroes of an ideal generated by "monomials". For foliations, due to the divergence of certain normal forms, the local models after reduction of singularities are formal one. Seidenberg's result was generalized in dimension $3$ by Cano-Cerveau (\cite{CanoCerveau}, non-dicritical case) and Cano (\cite{Cano}, general case). This reduction of singularities for foliation allows to prove Thom's conjecture about invariant local hypersurfaces, in dimension $2$ by Camacho-Sad (\cite{CamachoSad}), in dimension $3$ by Cano-Cerveau (\cite{CanoCerveau}) and in any dimension by Cano-Mattei (\cite{CanoMattei}):

\begin{theorem}
Any codimension $1$ germ of non dicritical holomorphic foliation has an invariant hypersurface.
\end{theorem}

Recall that there exist, in the dicritical case, codimension $1$ holomorphic foliations without local invariant hypersurface (\cite{Jouanolou}).

We give now the precise statement of reduction of singularities in dimension $3$ (an adapted version to a divisor is given in \cite{Cano} and \cite{CanoCerveau}).

\begin{theorem} 
Let $\mathcal F$ be a codimension $1$ holomorphic foliation over $X=\mathbb C^3,0$. Then there is a finite sequence of permissible blow-ups:
$$X=X(1) \displaystyle\mathop{\longleftarrow}^{^{\tiny{\quad \pi(1)}}} X(2) \displaystyle\mathop{\longleftarrow}^{^{\tiny{\quad \pi(2)}}} \cdots \displaystyle\mathop{\longleftarrow}^{^{\tiny{\ \ \pi(N)}}} X(N)$$
such that at each point $x\in X(N)$ the strict transform $\mathcal F_{N}$ of $\mathcal F$ by $\pi(N)\circ\cdots\circ \pi(1)$ either is non singular or has simple singularity.
\end{theorem}

The description of simple singularities can be given in terms of convenient adapted multiplicities (\cite{Cano}, \cite{CanoCerveau}). One of the difficulties in the theory is to give local models (like monomial equations in the case of hypersurfaces) for these simple singularities. After that, we can think that the simple singularities are given by  these normal forms.

So, let $\mathcal F$ be a codimension $1$ holomorphic foliation over $\mathbb C^3,0$. The foliation is said to be simple (\cite{CanoCerveau}) if and only if there exists a formal diffeomorphism $\phi \in \widehat{\mathrm{Diff}(\mathbb C^3,0)}$ (the formal completion of the group $\mathrm{Diff}(\mathbb C^3,0)$ of germs of holomorphic diffeomorphisms) such that $\phi^{-1} \mathcal F$ is given by one of the following meromorphic $1$-forms:

(1) $\frac{dx}{x} + \lambda\frac{dy}{y}\ \ ,\ \ \lambda \in \mathbb C\setminus\mathbb Q_{-}$,

(2) $\frac{dx}{x}+\Big(\varepsilon+\frac{1}{y^s}\Big) \frac{dy}{y }\ \ ,\ \  s\in \mathbb N\setminus\{0\}\ \ ,\ \   \varepsilon \in \mathbb  C$,

(3) $\frac{dx}{x} +\Big(\varepsilon + \frac{1}{(x^p y^q)^s}\Big) \Big(p \frac{dx}{x}+q\frac{dy}{y}\Big)$, $\mathrm{gcd}(p,q)=1$, $s\in \mathbb N\setminus\{0\}, \ \varepsilon \in \mathbb C$.

(4) $\alpha \frac{dx}{x} + \beta \frac{dy}{y}+\frac{dz}{z}\ \ ,\ \ \alpha\beta \neq 0$, $\alpha, \beta, \alpha/\beta \in \mathbb C\setminus\mathbb Q_{-}$,

(5) $\frac{dx}{x}+ \beta \frac{dy}{y}+(\varepsilon+\frac{1}{z^s}) dz,\ \ s\in \mathbb N\setminus\{0\},\ \ 0\neq \beta \in \mathbb C\setminus\mathbb Q_{-}$

(6) $\frac{dx}{x}+\beta \frac{dy}{y}+ \Big(\varepsilon+\frac{1}{(y^q z^q)^s}\Big) \Big(p\frac{dy}{y}+q\frac{dz}{z}\Big), s\in \mathbb N\setminus\{0\}$,
$gcd (p,q)=1, \ \ \varepsilon, \beta\in \mathbb C$

(7) $\frac{dx}{x}+\beta \frac{dy}{y}+ \Big(\varepsilon+\frac{1}{(x^p y^q z^r)^s}\Big) \Big(p\frac{dx}{x}+q\frac{dy}{y}+r \frac{dz}{z}\Big)$, $s\in \mathbb N\setminus\{0\}$,
$\mathrm{gcd} (p,q,r)=~1$, $\varepsilon$, $\beta\in \mathbb C$;

$x,y,z$ are linear coordinates in $\mathbb C^3$.

In some sense the seven types of previous $1$-forms describe the normal forms of generic meromorphic $1$-forms with normal crossing divisors of poles. Note that the forms (1), (2) and (3) are the models in dimension 2. Remark also that if $\mathcal F$ is a foliation over $M$ with only simple singularities, then $\mathcal F$ is given at each point $x\in M$ by a (formal) meromorphic $1$-form $\Omega_{x}$.

The previous $1$-form $\Omega_{x}$ is unique (up to multiplication by a complex number) except when $\mathcal F_{,x}$ has a non constant holomorphic first integral. The reason is that if two (formal) meromorphic closed $1$-forms $\Omega_{1}$ and $\Omega_{2}$ give the same foliation, then $\Omega_{2}= f.\Omega_{1}$ where $f$ is a (formal) meromorphic function. By differentiation, we see that, if $f$ is non-constant, $f$ is a (formal) meromorphic first integral. But it is easy to see that if $\mathcal F_{,x}$ is either non singular or simple, and if $\mathcal F_{,x}$ has a (formal) non constant first integral, then $\mathcal F_{,x}$ has a ordinary (formal) first integral (without poles); following Malgrange's singular Frobenius theorem $\mathcal F_{,x}$ has a non constant holomorphic first integral (\cite{Malgrange}).

\section{Foliations with simple singularities on $\mathbb P_{\mathbb C}^{3}$}

As we have seen, any codimension $1$ holomorphic foliation $\mathcal F$ on $\mathbb P_{\mathbb C}^{3}$ has a non trivial curve of singularities. We consider now special foliations on $\mathbb P_{\mathbb C}^{3}$.

\begin{proposition}\label{Prop:4}
Let $\mathcal F$ be a codimension $1$ holomorphic foliation on $\mathbb P_{\mathbb C}^{3}$. Suppose that there exists a component $\gamma$ $($of dimension $1)$ of the singular locus $\mathrm{Sing}\,\mathcal F$ such that: 

1) for any $x\in \gamma$, the germ $\mathcal F_{,x}$ has a simple singularity, in other words $\mathcal F$ is reduced along $\gamma$;

2) for any $x\in \gamma$, the germ $\mathcal F_{,x}$ has not a holomorphic non constant first integral.

Then $\mathcal F$ is given by a global closed meromorphic $1$-form. In particular $\mathcal F$ is transversally projective.
\end{proposition}

\begin{proof}
Assume at first that all the local models of $\mathcal F$ along $\gamma$ are given by meromorphic closed one-forms (that is the normalizing diffeomorphisms $\phi$ are convergent one). Then there exist a finite covering of $\gamma$ by open sets $U_{\alpha}$ and closed meromorphic $1$-forms $\Omega_{\alpha}$ defined on $U_{\alpha}$ such that $\mathcal F_{\vert U_\alpha}$ is given by $\Omega_{\alpha}$.
  If $\omega$ is a global  rational $1$-form on $\mathbb P_{\mathbb C}^{3}$ associated to $\mathcal F$ we have:
   $\Omega_{\alpha}=H_{\alpha}.\omega_{\vert U_{\alpha}}$, with $H_{\alpha}$ meromorphic on $U_{\alpha}$. On a non trivial intersection $U_{\alpha} \cap U_{\beta}$ we have:
  $$\Omega_{\alpha}=\lambda_{\alpha \beta} \Omega_{\beta}.$$
By hypothesis for a good choice of the covering, the cocycles $\lambda_{\alpha \beta}$ are constant.

The equality $H_{\alpha}=\lambda_{\alpha \beta} H_{\beta}$ gives by differentiation $\frac{dH_{\alpha}}{H_{\alpha}}=\frac{d H_{\beta}}{H_{\beta}}$ and the $\frac{d H_{\alpha}}{H_{\alpha}}$ define a closed meromorphic $1$-form $\omega'_{1}$ on a neighborhood of $\gamma$.

In fact, due to the possible divergence of the normalizing transformations, the local models $\Omega_{x}, x\in \gamma$, are not a priori convergent; a delicate study of the normalisations $\phi=\phi_{x}$ allows to study the dependence of $\Omega_{x}$ relative to $x\in \gamma$ and to see that the form $\omega'_{1}$ is a "formal meromorphic $1$-form along $\gamma$". The deep works \cite{Hironaka} and \cite{HironakaMatsumura} say that the form $\omega'_{1}$ is in fact the restriction of a global closed meromorphic $1$-form $\omega_{1}$ on $\mathbb P_{\mathbb C}^{3}$. This form $\omega_{1}$ has a decomposition as in Example~\ref{ex:4}:
$$\omega_{1}= \Sigma \lambda_{i} \frac{d f_{i}}{f_{i}}+dh$$
with $\lambda_{i} \in \mathbb C$, $f_{i}$ and $h$ rational functions.

Using the local construction of $\omega_{1}$ $(\omega_{1,x}=\frac{d H_{x}}{H_{x}}$ for $x\in \gamma$) we see that $\lambda_{i} \in \mathbb Z$ and $h\equiv 0$; so $\omega_{1}=\frac{dH}{H}$ for some rational function $H$. If we come back to the relations $\Omega_{\alpha}=H_{\alpha}.\omega_{\vert U_{\alpha}}$, we observe that 
$$\frac{dH}{H} \wedge \omega+d\omega=0$$
and the rational $1$-form $H.\omega$ is closed and defines the foliation $\mathcal F$. 
\end{proof}

\begin{remark}
In \cite{LinsNeto} Lins Neto uses that idea to glue local meromorphic closed Pfaffian forms to obtain a proximate result in a particular case.
\end{remark}

Now the next statement says that if a foliation $\mathcal F$ on $\mathbb P_{\mathbb C}^{3}$ has only simple singularities, then there exists a component $\gamma$ of $\mathrm{Sing}\,\mathcal F$ satisfying Proposition \ref{Prop:4}.

\begin{proposition}
Let $\mathcal F$ be a holomorphic codimension $1$ foliation over $\mathbb P_{\mathbb C}^3$. Suppose that all the non isolated singularities of $\mathcal F$ are simple; then the hypothesis of Proposition \ref{Prop:4} are satisfied.
\end{proposition} 

\begin{proof}
Take a generic linear planar section $i\colon\mathbb P_{\mathbb C}^2 \rightarrow \mathbb P_{\mathbb C}^3$ and denote by $\mathcal F_{0}$ the "restriction"  $i^{-1} \mathcal F$. It can be seen that all the singular points of $\mathcal F_{0}$ are simple (in the sense of dimension $2$). For such a singular point there is an index, the so-called Baum-Bott index. For a "hyperbolic" singular point $m_{0}$ of $\mathcal F_{0}$, that is given locally by a $1$-form of type:
$$\lambda_{1} x dy-\lambda_{2} y dx +\cdots$$
the Baum-Bott index is by definition $\mathrm{BB}(\mathcal F_{0}; m_{0})=\frac{(\lambda_{1}+\lambda_{2})^2}{\lambda_{1} \lambda_{2}}$. In the general case $\mathrm{BB}(\mathcal F_{0};m)$ is given by an explicit integral formula (\cite{Brunella}). There is a global index formula relating the local $\mathrm{BB}(\mathcal F_{0};m_0)$ to some special Chern-class, namely in the case of $\mathbb P_{\mathbb C}^2$:
$$\displaystyle\mathop{\Sigma}_{m_{0} \in\mathrm{Sing}\, \mathcal F_{0}}  \mathrm{BB}(\mathcal F_{0};m_0)=(n+2)^2$$
where $n$ is the degree of the foliation $\mathcal F_{0}$; this is the Baum-Bott formula (\cite{Brunella}).

Note  that if $m_{0}=i^{-1} (m)$ is a contact-singularity of $\mathcal F_{0}$, {\it i.e.} $m\notin \mathrm{Sing}\,\mathcal F$, then~$\mathcal F_{0}$ has a local holomorphic first integral of Morse type at $m_{0}$; as a consequence we have $\mathrm{BB}(\mathcal F_{0};m_{0})=0$. Suppose now that for all point $x$ belonging to any dimension~$1$ component $\gamma_{i}$ of  $\mathrm{Sing}\,\mathcal F$, the germ $\mathcal F_{,x} , x\in \gamma_{i}$, has a local non trivial holomorphic first integral. Then at each point $m_{0}\in\mathrm{Sing}\, \mathcal F_{0}$, the foliation is given by a $1$-form of the following type:
$$\omega_{, m_{0}}=p y dx+q x dy,\ \ \ p,q\in \mathbb N\setminus\{0\}, \ \textrm{gcd}(p,q)=1.$$
We see that the Baum-Bott index $\mathrm{BB}(\mathcal F_{0};m_{0})$ is negative, in contradiction with the Baum-Bott formula. 
\end{proof}

We are now able to give the main result of this paper:

\begin{theorem}\label{thm:7} 
Let $\mathcal F$ be a codimension $1$ holomorphic foliation on $\mathbb P_{\mathbb C}^3$. Suppose that all the non isolated singularities of $\mathcal F$ are simple. Then $\mathcal F$ is given by a closed rational $1$-form.
\end{theorem}

A standard extension result implies the:

 \begin{corollary} 
Let $\mathcal F$ be a codimension $1$ holomorphic foliation on $\mathbb P_{\mathbb C}^n, n\geq 3$. Suppose that there exists a linear section $i\colon\mathbb P_{\mathbb C}^3 \rightarrow \mathbb P_{\mathbb C}^n$, such that $i^{-1} \mathcal F$ is like in Theorem \ref{thm:7}. Then we have the same conclusion: $\mathcal F$ is given by a closed meromorphic $1$-form.
 \end{corollary}

 \begin{remark}
The structure of the ambiant space is important. For example a generic foliation on $\mathbb P_{\mathbb C}^2$ has only simple sigularities and is not given by a closed $1$-form; so there exist foliations on $\mathbb P_{\mathbb C}^2 \times \mathbb P_{\mathbb C}^1$ with only simple singularities which are not given by closed $1$-forms.
 \end{remark}

 \begin{remark}
 A consequence of Theorem \ref{thm:7} is the following: it is not possible to realize local dimension 3 simple singularities with divergent normalization, by a global foliation on $\mathbb P_{\mathbb C}^3$ having only simple singularities.
 \end{remark}
 
\subsection*{Acknowledgements}
 I wish to express my gratitude to Julie D\'eserti for her constant help.

\bibliographystyle{plain}
\bibliography{biblio}
\nocite{*}

\end{document}